\documentclass[a4paper,12pt,reqno]{amsart}
\usepackage{comment}
\usepackage[utf8]{inputenc}
\usepackage{latexsym}
\usepackage{amssymb}
\usepackage{amsmath}
\usepackage{enumerate}
\usepackage{xcolor}
\usepackage{url}
\usepackage{mathrsfs}
\usepackage{ulem}
\newtheorem{theorem}{Theorem}[section]
\newtheorem{lemma}[theorem]{Lemma}

\newtheorem{corollary}[theorem]{Corollary}
\newtheorem{example}[theorem]{Example}

\newtheorem{definition}[theorem]{Definition}
\newtheorem{question}[theorem]{Question}

\def\C{\mathbb C}
\def\Z{\mathbb Z}
\def\D{\mathbb D}
\def\T{\mathbb T}
\def\pb{\mathcal{DPB}}

\newcommand\ke{\operatorname{ker}}

\title[]{
Spectral decomposition of doubly power-bounded elements 
in Banach algebras
}

\author{
Osamu Hatori
}
\address{
{corresponding author\\
{ORCID  0000-0002-4338-1355}\\
Niigata University, 950-2181, Niigata, Japan.}
}

\email{hatori@math.sc.niigata-u.ac.jp
}

\author{
Shiho Oi
}
\address{
Department of Mathematics, Faculty of Science, 
Niigata University, 950-2181, Niigata, Japan.
}
\email{shiho-oi@math.sc.niigata-u.ac.jp
}

\keywords{
doubly power-bounded elements, spectrum, spectral decompositions, 
a theorem of Gelfand, a theorem of Koehler and Rosenthal
}
\subjclass[2020]{
47A10,47B06, 46B04
}

\begin{document}

\begin{abstract}
We establish a characterization of doubly power-bounded elements with finite spectrum in Banach algebras. In particular, we present a spectral decomposition for such elements, extending a classical theorem of Gelfand concerning doubly power-bounded elements with singleton spectrum. Furthermore, we generalize a theorem of Koehler and Rosenthal for doubly power-bounded elements to the setting of Banach algebras.  
In the final section, we are initiating a study to investigate whether the properties of doubly power-bounded elements can offer insight into the commutativity of Banach algebras.
\end{abstract}

\maketitle
\section{Introduction}
We investigate doubly power-bounded elements in complex Banach algebras with finite spectrum, continuing the line of study initiated in our previous paper \cite{oj1} on doubly power-bounded operators (we called just power-bounded operators in \cite{oj1}) on Banach spaces. This problem has its roots in a classical theorem by Gelfand \cite{gelfand}, which characterizes the identity operator as the doubly power-bounded operator whose spectrum is $\{1\}$. A detailed and insightful account of the developments related to this result is provided by Zem\'anek \cite{zemanek}.

Throughout this paper, unless otherwise stated, a Banach algebra is a complex one and 
$B$ denotes a unital Banach algebra with the unit element 
$e$. For an element 
$a\in B$, 
its spectrum is denoted by 
$\sigma_B(a)$, or simply 
$\sigma(a)$
 when the context is clear. 
 An idempotent in a Banach algebra is an element $a$ with $a^2=a$. 
 Given a complex {{(resp. real)}} Banach space 
$X$, we denote by 
$\mathfrak{B}(X)$ the {{complex (resp. real)}} Banach algebra of all bounded {{complex (resp. real)}} linear operators on 
$X$.
The identity operator on a Banach space (or, more generally, on a linear space) is denoted by $I$.

An invertible element 
$x\in B$ is said to be doubly power-bounded if
\begin{equation}\label{power}
\sup_{n\in \mathbb{Z}} \|x^n\| < \infty,
\end{equation}
where 
$\Z$ denotes the set of all integers. When 
$B={\mathfrak{B}}(X)$ {{for a complex Banach space $X$}}, doubly power-bounded elements are referred to as doubly power-bounded operators. In particular, every surjective complex-linear isometry on a complex Banach space is a doubly power-bounded operator.

In \cite{oj1}, we studied doubly power-bounded operators on arbitrary complex Banach spaces with isolated spectra. There, we proved a result of Koehler and Rosenthal for doubly power-bounded operators, showing that if a doubly power-bounded operator has an isolated point in its spectrum, then that point must be an eigenvalue, and its corresponding eigenspace has a complemented subspace. As a corollary, we obtained a spectral decomposition for doubly power-bounded operators with finite spectrum.

After revisiting a theorem of Ilis\v evi\'c about algebraic operators on Banach spaces in section \ref{s3}, we extend the above results to the more general setting of doubly power-bounded elements in Banach algebras in sections \ref{s4} and \ref{s5}. Our main result, Theorem \ref{3.2}, generalizes a classical theorem of Gelfand \cite[Satz 1]{gelfand}; see also \cite[Theorem 2.3, Corollary 3.3]{oj1}. 
In section \ref{s7}, we pose  Question \ref{7.2} concerning the commutativity of Banach algebras in terms of doubly power-bounded elements.

We denote the complex plane by $\C$. The unit circle in $\C$ is denoted by $\T$.
We apply the following two lemmata several times in the paper.
The first one is exhibited in \cite[Lemma 2.2]{oj1}. 
\begin{lemma}\label{2.2}
    Suppose that an invertible element $b\in B$ is doubly power-bounded.
Then $\sigma(b)\subset \T$. 
\end{lemma}
The next well known result appears, for example, in \cite[Theorem 1.2.8]{M}
\begin{lemma}\label{2.3} 
    Let $B_1$ be a closed subalgebra of $B$ which contains the unit of $B$. 
    Then $\partial \sigma_{B_1}(a)\subset \sigma_B(a)$, where $\partial \sigma_{B_1}(a)$ is the boundary of $\sigma_{B_1}(a)$.
\end{lemma}

\section{A theorem of Ilis\v evi\'c revisited}\label{s3}
Ilis\v evi\'c \cite[Proposition 2.4]{ilisevic} established a necessary and sufficient condition for a linear operator on a Banach space to be algebraic. It is worth noting that this result holds more generally for
linear operators on arbitrary vector spaces, whether real or complex ones, not just Banach spaces. In this paper, we provide a precise and detailed proof of this result. Furthermore, we extend their theorem to the broader context of Banach algebras.

The following result, including its proof, appears essentially in \cite[Theorem 5.-9D]{tay}. Throughout, the kernel of an operator is denoted by $\ke(\cdot)$. 
\begin{theorem}\label{1.1}
    Suppose that $L$ is a complex (resp. real) linear space. 
    Let $P_1,\dots, P_n$ be polynomials with complex (resp. real) coefficients. Suppose that any of the two $P_i$ and $P_j$ ($i\ne j$) are coprime. 
    Let $T\colon L\to L$ be a complex (resp. real) linear map. Then we have
    \begin{equation}\label{oplus}
    \ke P(T)=\oplus_{j=1}^n\ke P_j(T)
    \end{equation}
    for $P=\prod_{j=1}^nP_j$.
\end{theorem}

\begin{proof}
    Put $M=\ke P(T)$ and $M_j=\ke P_j(T)$, $1\le j\le n$. We prove the result by induction. 
    
    First, we prove it when $n=2$. In this case, there are polynomials $Q_1$ and $Q_2$ such that 
    \[
    Q_1P_1+Q_2P_2=1
    \]
    since $P_1$ and $P_2$ are coprime. Then we have
    \begin{equation}\label{1}
    Q_1(T)P_1(T)+Q_2(T)P_2(T)=I
    \end{equation}
    and
    \begin{multline}\label{2}
    Q_1(T)P_1(T)^2+Q_2(T)P_1(T)P_2(T)\\
    =Q_1(T)P_1(T)^2+Q_2(T)P(T)=P_1(T).
    \end{multline}
    We prove $M_1\cap M_2=\{0\}$. Suppose that $x\in M_1\cap M_2$. Then by \eqref{1} we have
    \[
    x=Q_1(T)P_1(T)x+Q_2(T)P_2(T)x=0
    \]
    since $P_1(T)x=P_2(T)x=0$. Thus $M_1\cap M_2=\{0\}$.  Next, we prove that $M=M_1\oplus M_2$. Suppose that $x\in M$. Put $y=Q_1(T)P_1(T)x$. Then by \eqref{2} we have
    \[
    P_1(T)x=P_1(T)Q_1(T)P_1(T)x+Q_2(T)P(T)x=P_1(T)y.
    \]
    Thus $x-y\in M_1$. We have
    \[
    P_2(T)y=P_2(T)Q_1(T)P_1(T)x=Q_1(T)P(T)x=0,
    \]
    hence $y\in M_2$. It follows that $x= x-y+y\in M_1 + M_2$. As $x$ is arbitrary, we infer that 
    $M=M_1+M_2$. As $M_1\cap M_2=\{0\}$, we see that $M=M_1\oplus M_2$. 

    Suppose that \eqref{oplus} holds for $n=k$. We prove \eqref{oplus} for $n=k+1$. As $P_i$ and $P_j$ are coprime for every $i\ne j$, we infer that $\prod_{j=1}^kP_j$ and $P_{k+1}$ are coprime. Then, by the first part and the assumption of induction, we have 
    \[
    \ke \prod_{j=1}^{k+1}P_j(T)=\ke \prod_{j=1}^kP_j(T)\oplus \ke 
    P_{k+1}(T)=\oplus_{j=1}^{k+1}\ke P_j(T).
    \]
    By induction, we have \eqref{oplus} for $n=k+1$.
\end{proof}
The following theorem corresponds to the case of linear operators on linear spaces in a theorem of Ilis\v evi\'c \cite[Proposition 2.4] {ilisevic} on bounded linear operators on Banach spaces. 
We define $\delta_{ij}=1$ if $i=j$, and $\delta_{ij}=0$ otherwise.
\begin{theorem}\label{1.2}
Suppose that $L$ is a complex (resp. real) linear space and $T\colon L\to L$ is a complex (resp. real) linear map.
Let $\lambda_1,\dots, \lambda_n$ be distinct complex (resp. real) numbers and $p_1,\dots, p_n$ be  complex (resp. real) linear operators on $L$ such that $p_ip_j=\delta_{ij}p_i$ for every $1\le i,j\le n$ and $\sum_{j=1}^np_j=I$. 
Then the following {\rm{(i)}} and {\rm{(ii)}} are equivalent:
\begin{itemize}
    \item[(i)] $T=\sum_{j=1}^n\lambda_jp_j$,
    \item[(ii)] $\prod_{j=1}^n(T-\lambda_jI)=0$ and $p_i=\prod_{j=1,j\ne i}^n\frac{T-\lambda_jI}{\lambda_i-\lambda_j}$ for every $1\le i\le n$ if $n>1$ and $p_1=I$ if $n=1$.
\end{itemize}
{{
In this case, $\sigma(T)\subset \{\lambda_1,\dots,\lambda_n\}$. 
In particular, if any $p_1,\dots,p_n$ is non-zero, then 
$\sigma(T)=\{\lambda_1,\dots, \lambda_n\}$.
}}
\end{theorem}
\begin{proof}
We only give a proof for $n>1$. The case where $n=1$ is trivial.
    (i)$\Rightarrow$(ii). For every $1\le j\le n$, we have
    \[
    T-\lambda_jI=\sum_{i=1}^n\lambda_ip_i-\lambda_j\sum_{i=1}^np_i=\sum_{i=1}^n(\lambda_i-\lambda_j)p_i
    =\sum_{i=1,i\ne j}^n(\lambda_i-\lambda_j)p_i.
    \]
    As $p_ip_j=\delta_{ij}p_i$ for every $1\le i,j\le n$, $\prod_{j=1}^n(T-\lambda_jI)$ is a linear combination of $p_1,\dots, p_n$. For any $1\le k\le n$, the coefficient of $p_k$ for $\prod_{j=1}^n(T-\lambda_jI)$ is nonzero only if the coefficient of $p_k$ in $T-\lambda_j$ for every $1\le j\le n$ is nonzero, which is impossible because the coefficient of $p_k$ for $T-\lambda_k I$ is zero.
    Hence 
    \[
    \prod_{j=1}^n(T-\lambda_jI)=0.
    \]
    For every $1\le i\le n$, we also have
    \[
    \prod_{j=1,j\ne i}^n(T-\lambda_jI)=\prod_{j=1,j\ne i}^n(\lambda_i-\lambda_j)p_i.
    \]
    Therefore we have 
    \[
    p_i=\prod_{j=1,j\ne i}^n\frac{T-\lambda_jI}{\lambda_i-\lambda_j}
    \]
    for every $1\le i\le n$.

    (ii)$\Rightarrow$(i). Letting $P(\lambda)=\prod_{j=1}^n(\lambda-\lambda_j)$, we have $P(T)=0$ by the first equality of (ii). 
    As $\lambda_1,\dots,\lambda_n$ are distinct numbers, monomials $\lambda-\lambda_i$ and $\lambda-\lambda_j$ are coprime for each $i\ne j$. 
    Then by Theorem \ref{1.1}, we have
    \[
    L=\ke P(T)=\oplus_{j=1}^n\ke (T-\lambda_jI).  
    \]
    Let $x\in L$. Then we have the expression $x=\sum_{j=1}^nx_j$, where $x_j\in \ke (T-\lambda_jI)$. The expression is unique since 
    $\ke (T-\lambda_jI)\cap \ke (T-\lambda_iI)=\{0\}$ for each $i\ne j$. Put $q_j\colon L\to L$ by $q_j(x)=x_j, x\in L$, where $x=\sum_{j=1}^nx_j$ for $x_j\in \ke (T-\lambda_jI)$. Then 
    \[
    \sum_{j=1}^n(T-\lambda_jI)q_j=0.
    \]
    Hence $T\sum_{j=1}^nq_j=\sum_{j=1}^n\lambda_jq_j$. As $\sum_{j=1}^nq_j=I$, we infer that 
    \[
    T=\sum_{j=1}^n\lambda_jq_j.
    \]
    By (i)$\Rightarrow$(ii), we have
    \[
    q_i=\prod_{j=1,j\ne i}^n\frac{T-\lambda_jI}{\lambda_i-\lambda_j}
    \]
    for every $1\le i\le n$. Hence,  $q_i=p_i$ for every $1\le i\le n$. 
    
Suppose that $T=\sum_{j=1}^n\lambda_jp_j$. 
We prove $\sigma(T)\subset \{\lambda_1,\dots,\lambda_n\}$. Let $\mu\in \C\setminus \{\lambda_1,\dots, \lambda_n\}$ be arbitrary. Then $T-\mu I=\sum_{j=1}^n(\lambda_j-\mu)p_j$ since $\sum_{j=1}^np_j=I$. As $\lambda_j-\mu\ne 0$ for every $j=1,\dots, n$, 
$\sum_{j=1}^n\frac{1}{\lambda_j-\mu}p_j$ is well defined and 
\[
(T-\mu I)\sum_{j=1}^n\frac{1}{\lambda_j-\mu}p_j=
(\sum_{j=1}^n\frac{1}{\lambda_j-\mu}p_j)(T-\mu I)=I,
\]
since $p_ip_j=\delta_{ij}p_i$ for every $1\le i,j\le n$; $\mu\not\in \sigma (T)$. As $\mu\in \C\setminus \{\lambda_1,\dots,\lambda_m\}$ is arbitrary,  we have 
$\sigma(T)\subset \{\lambda_1,\dots, \lambda_n\}$. 
Suppose further that any $p_1,\dots, p_n$ is non-zero. We prove $\{\lambda_1,\dots, \lambda_n\}\subset \sigma(T)$. Suppose that $\lambda_k\not\in \sigma(T)$ for some $1\le k\le n$; $T-\lambda_kI$ is invertible. Then 
$\prod_{j=1,j\ne k}^n(T-\lambda_j I)=0$ due to the first equality of (ii). Then by the second one, we have 
\[
p_k=\prod_{j=1,j\ne k}^n\frac{T-\lambda_jI}{\lambda_k-\lambda_j}=0,
\]
which is a contradiction. Hence, $\{\lambda_1,\dots, \lambda_n\}\subset \sigma(T)$. We conclude that  $\{\lambda_1,\dots, \lambda_n\}= \sigma(T)$.
\end{proof}
As a straightforward application of Theorem \ref{1.2}, we recover the following result of Ilis\v evi\'c \cite[Proposition 2.4]{ilisevic}.
\begin{corollary}\label{1.3}
    Suppose that $X$ is a complex (resp. real) Banach space and $T\colon X\to X$ is a bounded complex (resp. real) linear operator. 
    Let $\lambda_1,\dots, \lambda_n$ be distinct complex (resp. real) numbers and $p_1,\dots, p_n$ be a complex (resp. real) linear operator on $L$ such that $p_ip_j=\delta_{ij}p_i$ for every $1\le i,j\le n$ and $\sum_{j=1}^np_j=I$. 
Then the following {\rm{(i)}} and {\rm{(ii)}} are equivalent:
\begin{itemize}
    \item[(i)] $T=\sum_{j=1}^n\lambda_jp_j$,
    \item[(ii)] $\prod_{j=1}^n(T-\lambda_jI)=0$ and $p_i=\prod_{j\ne i}\frac{T-\lambda_jI}{\lambda_i-\lambda_j}$ for every $1\le i\le n$ if $n>1$ and $p_1=I$ if $n=1$.
\end{itemize}
In this case, $\sigma(T)\subset \{\lambda_1,\dots,\lambda_n\}$. 
In particular, if any $p_1,\dots,p_n$ is non-zero, then 
$\sigma(T)=\{\lambda_1,\dots, \lambda_n\}$.
\end{corollary}
Note that although boundedness of the operators $p_1,\dots,p_n$ in Corollary \ref{1.3} is not assumed explicitly, it automatically follows from condition (ii).
\begin{definition}
{{Let $A$ be a unital complex (resp. real) Banach algebra.}}
  For any $x\in A$, the multiplication operator $S_x\colon A\to A$ is given by $S_x(y)=xy$ for $y\in A$. Define $S\colon A \to \mathfrak{B}(A)$ by $S(x)=S_x$ for $x\in A$.  
\end{definition}
\begin{lemma}\label{2.1}
{{Let $A$ be a unital complex (resp. real)  Banach algebra. }} Then $S\colon A\to \mathfrak{B}(A)$ gives an isometrical algebra isomorphism from $A$ onto $S_A=\{S_x\colon x\in A\}$. 
\end{lemma}
The proof is routine and is omitted. 
Note that $S_A=\{S_x\colon x\in A\}$ is a unital closed subalgebra of $\mathfrak{B}(A)$, where $S_e=I$. 

We present a theorem of Illis\v evi\'c in the context of Banach algebras.
\begin{corollary}\label{3.4}
    Suppose that $A$ is a unital complex (resp. real) Banach algebra and $b\in A$.
    Let $\lambda_1,\dots, \lambda_n$ be distinct complex (resp. real) numbers and $p_1,\dots, p_n\in A$ satisfy $p_ip_j=\delta_{ij}p_i$ for every $1\le i,j\le n$ and $\sum_{j=1}^np_j=e$. 
    Then the following {\rm{(i)}} and {\rm{(ii)}} are equivalent:
\begin{itemize}
    \item[(i)] $b=\sum_{j=1}^n\lambda_jp_j$,
    \item[(ii)] $\prod_{j=1}^n(b-\lambda_je)=0$ and $p_i=\prod_{j\ne i}\frac{b-\lambda_je}{\lambda_i-\lambda_j}$ for every $1\le i\le n$ if $n>1$ and $p_1=e$ if $n=1$.
\end{itemize}
In this case, $\sigma(b)\subset \{\lambda_1,\dots,\lambda_n\}$. 
In particular, if any $p_1,\dots,p_n$ is non-zero, then 
$\sigma(b)=\{\lambda_1,\dots, \lambda_n\}$.
\end{corollary}
\begin{proof}
    Applying the map $S\colon A\to \mathfrak{B}(A)$, 
    we can rewrite 
    (i) and (ii)  by 
    \begin{itemize}
     \item[(i)'] $S_b=\sum_{j=1}^n\lambda_jS_{p_j}$,
    \item[(ii)'] $\prod_{j=1}^n(S_b-\lambda_jS_e)=0$ and $S_{p_i}=\prod_{j\ne i}\frac{S_b-\lambda_jS_e}{\lambda_i-\lambda_j}$ for every $1\le i\le n$ if $n>1$ and $S_{p_1}=I$ if $n=1$. 
    \end{itemize}
By Corollary \ref{1.3}, we have (i)' and (ii)' are equivalent. 
Hence, (i) and (ii) are equivalent.
In this case, $\sigma_{\mathfrak{B}(A)}(S_b)\subset \{\lambda_1,\dots,\lambda_n\}$ by Corollary \ref{1.3}. By Lemmata \ref{2.2}, \ref{2.3} and \ref{2.1}, we have 
$\sigma(b)=\sigma_{S_A}(S_b)\subset \{\lambda_1,\dots,\lambda_n\}$.
Furthermore, if any $p_1,\dots,p_n$ is non-zero, then 
$\sigma_{\mathfrak{B}(A)}(S_b)=\{\lambda_1,\dots, \lambda_n\}$ 
by Corollary \ref{1.3}. By Lemmata \ref{2.2}, \ref{2.3} and \ref{2.1}, we have $\sigma(b)=\sigma_{S_A}(S_b)=\{\lambda_1,\dots, \lambda_n\}$.
\end{proof}
\section{A theorem of Koehler and Rosenthal for doubly power-bounded elements in Banach algebras}\label{s4}
The following is a version of a theorem of Koehler and Rosenthal \cite{KR,oj1} for Banach algebras.

\begin{theorem}\label{3.1}
    Let $B$ be a unital complex Banach algebra. Suppose that an invertible element $b\in B$ is doubly power-bounded. 
    Suppose that $\lambda$ is an isolated point in $\sigma(b)$. 
Then there exists an idempotent $p\in B$ such that
\[
bp=\lambda p.
\]
\end{theorem}
\begin{proof}
Since $S\colon B\to S_B$ is defined by $S(x)=S_x$ and is an isometrical isomorphism by Lemma \ref{2.1}, 
we have $\sigma_{S_B}(S_b)=\sigma(b)$ and $\sup_{n\in \Z}\|S_b^n\|=\sup_{n\in \Z}\|b^n\|<\infty$. As $\lambda$ is an isolated point in $\sigma(b)=\sigma_{S_B}(S_b)$, 
$\lambda$ is in the boundary of $\sigma_{S_B}(S_b)$.  
Since $S_B$ is a unital closed subalgebra of $\mathfrak{B}(B)$, Lemma \ref{2.3} implies that $\lambda \in \sigma_{\mathfrak{B}(B)}(S_b)$. Moreover, since $\sigma_{\mathfrak{B}(B)}(S_b)\subset \sigma_{S_B}(S_b)$, $\lambda$ is an isolated point of $\sigma_{\mathfrak{B}(B)}(S_b)$.
Suppose that $\Gamma$ is a Cauchy contour in the resolvent set $\C\setminus \sigma_{S_B}(S_b)$ of 
$S_b$ around $\lambda$ separating $\lambda$ from $\sigma_{S_B}(S_b)\setminus \{\lambda\}$. 
As $\C\setminus \sigma_{S_B}(T_b)\subset \C\setminus \sigma_{\mathfrak{B}(B)}(S_b)$, we have that $\Gamma$ is also a Cauchy contour in the resolvent set 
$\C\setminus \sigma_{\mathfrak{B}(B)}(S_b)$. As $\lambda$ is isolated, we may suppose that $\Gamma$ separates $\lambda$ from $\sigma_{\mathfrak{B}(B)}(S_b)$. 
Let $Q$ be the Riesz projection corresponding to $\lambda$ defined by 
    \[
    Q=\frac{1}{2\pi i}\int_\Gamma(S_b-wI)^{-1}dw.
    \]
    Note that $Q\in \mathfrak{B}(B)$. 
    Please refer to \cite{GGK} for properties of the Riesz projection.
By \cite[Theorem 3.2]{oj1} 
\begin{equation}\label{eq1}
    S_b|Q(B)=\lambda I_{Q(B)}, 
\end{equation} where $I_{Q(B)}$ is the identity map on $Q(B)$. 
As $(S_b-wI)^{-1}\in S_B$ for $w\in \Gamma$, we have $Q\in S_B$.
By the definition of $S_B$, there is $p\in B$ such that $S_p=Q$. As $Q$ is a projection in the sense that $Q=Q^2$, we have $p=p^2$ by Lemma \ref{2.1}; $p$ is an idempotent in $B$. Rewriting \eqref{eq1} we have
\[
bp=\lambda p.
\]
\end{proof}
\section{Representation of doubly power-bounded elements with finite spectrum in Banach algebras: a generalization of a theorem of Gelfand}\label{s5}
Following the definition in the case of operators, we define the Riesz projections in Banach algebras. 
Recall that $B$ denotes a unital complex Banach algebra with the unit $e$. 
\begin{definition}
Let 
$a\in B$. Suppose that $\lambda$ is an isolated point in $\sigma(a)$. We call
\[
p=\frac{1}{2\pi i}\int_\Gamma(a-wI)^{-1}dw 
\]
the Riesz projection corresponding to $\lambda$, where 
$\Gamma$ is a Cauchy contour in the resolvent set $\C\setminus \sigma(a)$ around $\lambda$ separating $\lambda$ from ${\sigma{(a)}}\setminus \{\lambda\}$. 
\end{definition}
Note that the Riesz projection does not depend on the choice of a Cauchy contour. 

The following is 
a characterization of doubly power-bounded elements with finite spectrum, which is 
a generalization of a theorem of Gelfand \cite[Satz 1]{gelfand}.  Note that the corresponding result for operators is exhibited in \cite{oj1}. 
\begin{theorem}\label{3.2} 
Suppose that $b\in B$ is invertible and {{$\{\lambda_1,\dots, \lambda_n\}\subset \C$}}. The following are equivalent.
\begin{itemize}
    \item[(i)] $b$ is doubly power-bounded {{and $\sigma(b)=\{\lambda_1,\dots,\lambda_n\}$}};
    \item[(ii)] $\{\lambda_1,\dots,\lambda_n\}\subset \T$, and there exist {{non-zero}} idempotents $p_1,\dots, p_n\in B$ such that $\sum_{j=1}^np_j=e$ and $p_ip_j=\delta_{ij}p_i$ for $1\le i,j\le n$ which satisfy that 
    \[
    b=\sum_{j=1}^n\lambda_jp_j;
    \]
    \item[(iii)] $\{\lambda_1,\dots,\lambda_n\}\subset \T$, and $b$ is an algebraic element in the sense that 
    \[
    \prod_{j=1}^n(b-\lambda_je)=0,
    \]
    and we have
    \[
    0\ne p_i=\prod_{j=1,j\ne i}^n\frac{b-\lambda_je}{\lambda_i-\lambda_j}
    \]
    for $1\le i\le n$ if $n>1$ and $p_1=e$ if $n=1$.
\end{itemize}
In this case, if $n>1$, then $p_i$ is the Riesz projection corresponding to $\lambda_i$ for every $1\le i\le n$. 
\end{theorem}
\begin{proof}
We prove (i)$\Rightarrow$(ii).
First, Lemma \ref{2.2} ensures that $\sigma(b)\subset \T$.
Recall that $S\colon B\to \mathfrak{B}(B)$ is defined as $S(x)=S_x$ for $x\in B$, where $S_x(y)=xy$ for $y\in B$.
    In a similar way to the proof of Theorem \ref{3.1}, we have 
    $\sigma(b)=\sigma_{S_B}(S_b)$. As each $\lambda_i$, $1\le i\le n$,  is in the boundary of $\sigma_{S_B}(b)$, we have that $\sigma_{S_B}(S_b)=\sigma_{\mathfrak{B}(B)}(S_b)$ by Lemma \ref{2.3}. Then by \cite[Corollary 3.3]{oj1}, we observe that 
    \begin{equation}\label{eq2}
    S_b=\sum_{j=1}^n\lambda_jQ_j,
    \end{equation}
    where $Q_j$ is the Riesz projection corresponding to $\lambda_j$ for every $1\le j \le n$ such that $I=\sum_{j=1}^nQ_j$ and $Q_iQ_j=0$ for every pair $(i,j)$ with $i\ne j$. In a similar way to the proof of Theorem \ref{3.1}, we see that $Q_j\in S_B$ for every $j=1,\dots, n$. 
    Let $p_j$ be a Riesz projection corresponding to $\lambda_j$.
    As $S$
    is an isometrical algebra isomorphism (Lemma \ref{2.1}), we see that 
    $S_{p_i}=Q_j$ for every $1\le i\le n$ and 
    $\sum_{j=1}^np_j=e$ and $p_ip_j=\delta_{ij}p_i$ for every $1\le i,j\le n$. Rewriting \eqref{eq2} we observe
    \[
    b=\sum_{j=1}^n\lambda_jp_j.
    \]

    A proof of (ii)$\Rightarrow$(i).  Since 
    \[
    \|b^n\|=\|\sum_{j=1}^n\lambda_j^np_j\|\le \sum_{j=1}^n\|p_j\|
    \]
    for every $n\in \Z$, we have that $b$ is doubly power-bounded.  By Corollary \ref{3.4} we have that $\sigma(b)=\{\lambda_1,\dots, \lambda_n\}$. 
    
    By Corollary \ref{3.4} we have (ii) implies (iii). 
    In particular,  we also have $p_i=\prod_{j=1,j\ne i}^n\frac{b-\lambda_je}{\lambda_i-\lambda_j}$ for $1\le i\le n$ if $n>1$ and $p_1=I$ if $n=1$.

    (iii)$\Rightarrow$(ii). Suppose that (iii) is satisfied. Letting  
    $p_i=\prod_{j=1,j\ne i}^n\frac{b-\lambda_je}{\lambda_i-\lambda_j}$ for $1\le i\le n$ if $n>1$, we obtain (ii) by Corollary \ref{3.4}. If $n=1$, then (ii) is trivial.   
\end{proof}

\begin{corollary}
    Suppose that an invertible element $u\in B$ satisfies 
    \[
    \|u^{-1}\|=\|u\|=1.
    \]
    Suppose that $\sigma(u)=\{\lambda_1,\dots, \lambda_n\}$.  
Then there exist idempotents $p_1,\dots, p_n\in B$ such that  $\sum_{j=1}^np_j=e$ and $p_ip_j=\delta_{ij}p_i$ for $1\le i,j\le n$ which satisfy
\[
u=\sum_{j=1}^n\lambda_jp_j.
\]
\end{corollary}

\begin{proof}
    We prove that $\|u^n\|=1$ for every positive integer $n$ by induction. Suppose that $\|u^k\|=1$ for a positive integer $k$. Then 
    \begin{multline*}
    1=\|u^k\|=\|u^{-1}u^{k+1}\|\le \|u^{-1}\|\|u^{k+1}\|\\
    =\|u^{k+1}\|\le \|u\|\|u^k\|=\|u^k\|=1.
    \end{multline*}
Thus $\|u^{k+1}\|=1$. Therefore we have $\|u^n\|=1$ for every positive integer $n$. We have $\|u^{-n}\|=1$ for every positive integer $n$ in the same way. It follows that $u$ is doubly power-bounded.  Then by Theorem \ref{3.2} we have the conclusion.    
\end{proof}

\begin{corollary}
       Suppose that $a\in B$ is periodic in the sense that $a^m=e$ for a positive integer $m$. 
  Then there exist idempotents $p_1,\dots, p_m \in B$ with $\sum_{j=1}^m p_j=e$ and $p_ip_j=\delta_{ij}p_i$ for $1\le i,j\le m$ which satisfy
\[
a=\sum_{j=1}^m e^{2\pi ji/m} p_j.
\] 
 Note that $p_k=0$ if $e^{2\pi ki/m} \notin \sigma(a)$.
\end{corollary}
\begin{proof} 
The element $a$ is obviously doubly power-bounded. 
    By the spectrum mapping theorem, we have 
    $\sigma(a)\subset \{e^{2\pi ki/m}\colon 1\le k\le m\}$. 
    Put $\sigma(a)=\{e^{2\pi k_1i/m},\dots, e^{2\pi k_ni/m}\}$. 
    Then Theorem \ref{3.2} asserts that there exist idempotents $p_{k_1},\dots, p_{k_n}\in B$ with $\sum_{l=1}^np_{k_l}=e$ and $p_{k_l}p_{k_s}=\delta_{ls}p_{k_l}$ for $1\le l,s\le n$ which satisfy 
    \begin{equation}\label{541}
    a=\sum_{l=1}^ne^{2\pi k_li/m}p_{k_l}.
    \end{equation}
    If $e^{2\pi ki/m}\not\in \sigma(a)$, then put $p_k=0$. Rewriting \eqref{541}, we have the conclusion.
\end{proof}
\section{
Doubly power-bounded elements characterize commutativity?
}\label{s7}
We denote the set of all doubly power-bounded elements in $B$ by $\pb$.
Let $U=\{a\in B^{-1}\colon \|a^{-1}\|=\|a\|=1\}$. The set $U$ coincides with the unitary group for a unital $C^*$-algebra.  
By a simple calculation, we have
\[
U\subset \{a^{-1}ua\colon a\in B^{-1}, u\in U\}\subset \pb\subset \{a\in B\colon \sigma(a)\subset \T\},
\]
for a general unital Banach algebra. If $B$ is commutative, then $U=\{a^{-1}ua\colon a\in B^{-1}, u\in U\}$. 
If $B$ is a uniform algebra, then $U=\{a\in B\colon \sigma (a)\subset \T\}$. (Suppose that $\sigma(a)\subset \T$. By the spectral mapping theorem, we have that $\sigma(a^{-1})\subset \T$. As the spectral norm coincides with the original one for a uniform algebra, we infer that $\|a\|=1$ and $\|a^{-1}\|=1$. 
Conversely, suppose that $\|a\|=\|a^{-1}\|=1$. Then we have $\sigma(a)\subset \D$, where $\D$ denotes $\{z\in \C\colon |z|\le 1\}$. We also have that $\sigma(a^{-1})\subset \D$. By the spectral mapping theorem, we have $\sigma(a)\subset \{z\in \C\colon |z|\ge1\}$. Thus we see that $\sigma(a)\subset \T$.)
Hence, $U=\pb=\{a\in B\colon \sigma(a)\subset \T\}$ if $B$ is a uniform algebra.

For a locally compact group $G$, we denote $B(G)$ by the Fourier-Stieltjes algebra on $G$, which is defined as the linear span of all continuous positive definite functions on $G$ and can be identified with the dual space of the group $C^*$-algebra $C^*(G)$. 
 We denote $\widehat{G}$ the set of all continuous and multiplicative map $\gamma\colon G\to \T$, that is, $\widehat{G}$ denotes the set of all continuous characters on $G$. 
 A characterization of doubly power bounded elements in $B(G)$ can be reformulated by a theorem of Kaniuth and \"Ulger \cite[Theorem 4.5]{ku}; a function $u\in B(G)$ is doubly power bounded if and only if there is a finite number of open cosets $F_1,\dots, F_m$ of $G$, disjoint open subgroups $H_1,\dots, H_m$ of $G$, $a_1,\dots, a_m\in G$  with $F_k=a_kH_k$ and $\cup_{k=1}^mF_k=G$, a character $\gamma_k$ on $H_k$ for every $k=1,\dots, m$, and unimodular constants $\lambda_1,\dots, \lambda_m$ such that 
 $u=\sum_{k=1}^m\lambda_k1_{F_k}L_{a_k}\gamma_k$. 
 Restating \cite[Corollary 4.6]{ku} we have that if $G$ is connected, then $\pb=\{\alpha\gamma\colon \alpha\in \T,\, \gamma\in \widehat{G}\}=U$, where $\widehat{G}$ denotes the set of all characters on $G$. 
 Let  $B(\T)$ be the Wiener algebra. 
 Note that the M\"obius transformation $f(z)=(2z-1)/(2-z)$ in $B(\T)$ satisfies $\sigma(f)= \T$ and $\|f\|=2$  since $f(z)=-\frac12+\frac{3}{2}\sum_{n=1}^\infty\frac{1}{2^n}z^n$. Hence, we have $U=\pb\subsetneq\{f\in B(\T):\sigma(f)\subset \T\}$.
 In general, we have the following without assuming the connectivity of $G$.
\begin{theorem}
    Suppose that $G$ is a locally compact group. Then, 
    \[
    \{u\in \pb\colon\|u\|=1\}
    =\{\alpha \gamma \colon \alpha\in \T,\,\,\gamma\in \widehat{G}\}
    =U.
    \]
    \end{theorem}
\begin{proof}
Put $\T\widehat{G}=\{\alpha\gamma\colon \alpha\in \T, \,\,\gamma\in \widehat{G}\}$ and recall that $U=\{u\in B(G)^{-1}\colon \|u\|=\|u^{-1}\|=1\}$. 
Obviously, $U\subset \{u\in \pb\colon\|u\|=1\}$. 
    Suppose that $\alpha\in \T$ and $\gamma \in \widehat{G}$. 
    We have
    \[
    \|\gamma\|=\sup_{f\in C^*(G), \|f\|\le 1}\left|\int\gamma(x)f(x)d\mu (x)\right|\le 1.
    \]
    By the theory of commutative Banach algebras, we also have
    \[
    \|\gamma\|\ge r(\gamma)\ge |\gamma(e)|=1, 
    \]
    where $r(\cdot)$ is the spectral radius and $\gamma(e)=1$. Thus $\|\alpha\gamma\|=1$, hence $\|(\alpha\gamma)^{-1}\|=1$. It follows that $\T\widehat{G}\subset U$. 

    Let $u\in \pb$ be such that $\|u\|=1$. 
    By Lemma \ref{2.2}, $\sigma(u)\subset \T$. 
    By a theorem of Eymard \cite[Lemma 2.14]{ey}, there exists a unitary representation $\pi$ of $G$ and $\xi, \eta\in H_{\pi}$ with $\|\xi\|=\|\eta\|=1$ such that 
    \[
    u(x)=\langle\pi(x)\xi,\eta\rangle,\quad x\in G.
    \]
    By the Cauchy-Schwarz inequality, we have
    \[
    1=|u(x)|=|\langle \pi(x)\xi,\eta\rangle|\le \|\pi(x)\xi\|\|\eta\|\le \|\xi\|\|\eta\|=1.
    \]
    It follows that there exists a complex number $\alpha(x)$ such that 
    $\pi(x)\xi=\alpha(x)\eta$, hence we have
    \[
u(x)=\langle\pi(x)\xi,\eta\rangle=\langle\alpha(x)\eta,\eta\rangle=\alpha(x)
    \]
    for every $x$. Thus 
    \begin{multline}\label{720}
    u(xy)=\langle\pi(xy)\xi,\eta\rangle=\langle\pi(x^{-1})^{-1}\alpha (y)\eta,\eta\rangle 
    \\
    =\alpha(y)\langle\pi(x^{-1})^{-1}\eta,\eta\rangle=
    \alpha\langle\alpha(x^{-1})^{-1}\xi,\eta\rangle=
    \alpha(y)\alpha(x^{-1})^{-1}\langle\xi,\eta\rangle
    \end{multline}
    for every pair $x,y\in G$. 
    Then we have
    \begin{equation}\label{721}      u(e)=u(e^{-1}e)=\alpha(e)\alpha(e)^{-1}\langle\xi,\eta\rangle=\langle\xi,\eta\rangle
    \end{equation}
    and 
    \begin{equation}\label{722}    \alpha(x)=u(x)=u(xe)=\alpha(e)\alpha(x^{-1})^{-1}\langle\xi,\eta\rangle
    \end{equation}
    for every $x\in G$. 
    Suppose that $u(e)=1$ first. Then $\alpha(e)=u(e)=1$, and 
    $1=u(e)=\langle \xi,\eta\rangle$ by \eqref{721}. 
    Thus, by \eqref{722} we get $\alpha(x^{-1})=\alpha(x)^{-1}$. 
    By \eqref{720}, we have $u(xy)=u(x)u(y)$ for every pair $x,y\in G$, which ensures that $u$ is a character on $G$; $u\in \T\widehat{G}$. Suppose that $u(e)$ need not be $1$. Put $v=\overline{u(e)}u$. Since $|u|=1$ on $G$, we have $v(e)=1$ and $v\in \pb$. It follows by the previous part that $v\in \T\widehat{G}$, hence $u=u(e)v\in \T\widehat{G}$.  
\end{proof}
Suppose that $B$ is commutative. Then $\pb$ is closed under multiplication since for any $a, b\in \pb$, and an integer $n$, we have $\|(ab)^n\|\le\|a^n\|\|b^n\|$. 
On the other hand, $\pb$ can be closed under multiplication even if $B$ is noncommutative, as the following example shows.
\begin{example}
    Suppose that $n$ is a positive integer greater than   $2$. 
    Let $B$ be a subalgebra of the algebra of all $n\times n$ complex matrices which consists of all upper triangle matrices with identical diagonal entries. Then 
    \[
    \pb=\{\lambda I\colon \lambda\in \T\},
    \]
    where $I$ is the identity matrix. 
    The reason is as follows. Suppose that $M=\lambda I+ N\in \pb$, where $N$ is the nilpotent part of $M$. As the $\sigma(M)\subset \T$, $|\lambda|=1$. For a positive integer $m\ge n$, $M^m=\lambda^m I+ m\lambda^{m-1}N+\cdots + \lambda^{n-1}
    \begin{pmatrix}
m \\
m-n+1 \\
\end{pmatrix}
    N^{n-1}$ since $N^k$ is the zero matrix for $k\ge n$. We easily see that $\|M^n\|\to \infty$ unless $N$ is the zero matrix. It follows that $M=\lambda I$. Conversely, $\lambda I\in \pb$ for $\lambda\in \T$ is clear. Note that $B$ is neither commutative nor semisimple.
\end{example}
\begin{question}\label{7.2}
    Suppose that $B$ is semisimple and $\pb$ for $B$ is closed under multiplication. Does it follow that $B$ is commutative? How about for the case of a unital 
    $C^*$-algebra?
\end{question}

We provide a partial answer to the question. 
Recall that a standard operator algebra is a subalgebra of 
$B(X)$ 
 containing all finite-rank bounded operators on a complex Banach space $X$.
The Toeplitz algebra (generated by a unilateral shift) and the Laurent algebra (generated by a bilateral shift) are typical examples. 
Standard $C^*$-algebras play an essential role in the Brown–Douglas–Fillmore (BDF) theory since extensions of the algebra of compact operators by a commutative $C^*$-algebra is the central concept in the theory. 
A standard $C^*$-algebra on a Hilbert space $H$ contains every finite rank bounded operator on $H$. 
Hence, it is not commutative if the dimension of $H$ is greater than 1. We have the following.

\begin{theorem}\label{7.2}
    Suppose that $A$ is a standard unital $C^*$-algebra such that $A\subset B(H)$ with a Hilbert space $H$ of dimension greater than 1. 
    Then, $\pb$ is not closed under multiplication. 
\end{theorem}
\begin{proof}
Suppose that $\pb$ is closed under multiplication. 
Since $A$ contains every finite-rank projection, we infer that the commutant of $A$, $\{x\in B(H)\colon xa=ax\}$, coincides with $\C e$. Then \cite[Theorem 4.1.12]{M} asserts that $A$ acts irreducibly on $H$. 
    It is well known that a unital $C^*$-algebra is generated by unitaries. As the dimension of $H$ is greater than 1, there exists linearly independent $x,y\in H$ with $\|x\|=\|y\|=1$ and a unitary element $u\in U$ such that $u(x)=y$. 
 By Kadison's transitivity theorem (cf. \cite[Theorem 5.2.2]{M}), there exists $a\in A$ such that $a(x)=x$ and $a(y)=2y$. Letting $k=2\|a\|\ge 4$,  $b=(a+ke)/(1+k)\in A$ is invertible and $b(x)=x$, $b(y)=(2+k)/(1+k)y$.   
 As $u$ is an surjective isometry on $H$, we have $u\in \pb$, hence, $b^{-1}ub\in \pb$. As we assumed that $\pb$ is closed under muplitplication, we have $u^{-1}b^{-1}ub\in \pb$. Thus $\sigma(u^{-1}b^{-1}ub)\subset \T$ by Lemma \ref{2.2}.  
 By a direct calculation, we infer that $u^{-1}b^{-1}ub(x)=\frac{1+k}{2+k}x$. Thus, $\frac{1+k}{2+k}\in \sigma(u^{-1}b^{-1}ub)$, which is against Lemma \ref{2.2} since $\frac{1+k}{2+k}$ is not unimodular.
\end{proof}

\subsection*{Acknowledgments}
The second-named author was supported by JSPS KAKENHI Grant Numbers JP24K06754. 
This version of the article has been accepted for publication, after peer review but is
not the Version of Record and does not reflect post-acceptance improvements, or any corrections. The Version of Record is
available online at: http://dx.doi.org/[10.1007/s44146-026-00242-9]. Use of this Accepted Version is subject to the publisher’s Accepted
Manuscript terms of use https://www.springernature.com/gp/open-research/policies/accepted-manuscript-terms”.
\subsection*{Declaration}
The authors used ChatGPT (OpenAI) to assist with English-language editing. After using this tool, the authors reviewed and edited the content as needed and take full responsibility for the 
content of the publication.
\subsection*{Data Availability Statement} Data sharing is not applicable to this article as
no datasets were generated or analyzed during the current study.
\subsection*{Conflict of interest} 
On behalf of all authors, the corresponding author states that there is no conflict of interest.

\end{document}